\documentclass[letterpaper,11pt,reqno]{amsart}

\usepackage{amssymb}
\usepackage[numbers]{natbib}
\usepackage{latexsym}
\usepackage{amsbsy}
\usepackage{amsfonts}
\usepackage[hidelinks]{hyperref}
\usepackage{graphicx}
\usepackage{enumerate}
\usepackage{enumitem}
\usepackage{mathtools}
\usepackage[margin=1.3in]{geometry}
\usepackage{color}
\usepackage{tikz}

\usepackage{etoolbox} 

\makeatletter
\def\@noindentfalse{\global\let\if@noindent\iffalse}
\def\@noindenttrue {\global\let\if@noindent\iftrue}
\def\@aftertheorem{%
  \@noindenttrue
  \everypar{%
    \if@noindent%
      \@noindentfalse\clubpenalty\@M\setbox\z@\lastbox%
    \else%
      \clubpenalty \@clubpenalty\everypar{}%
    \fi}}

\def\note#1{\par\smallskip%
\noindent\kern-0.01\hsize%
\setlength\fboxrule{0pt}\fbox{\setlength\fboxrule{0.5pt}\fbox{%
\llap{$\boldsymbol\Longrightarrow$ }%
\vtop{\hsize=0.98\hsize\parindent=0cm\small\rm #1}%
\rlap{$\enskip\,\boldsymbol\Longleftarrow$}
}}%
}

\DeclareMathOperator\var{Var}
\DeclareMathOperator\thaa{th}
\DeclareMathOperator\inv{inv}

\DeclareMathOperator\Lip{Lip}
\DeclareMathOperator\BC{BC}
\DeclareMathOperator\I{I}

\newcommand\dw{d_{\mathrm{W}}}

\let\original@left\left
\let\original@right\right
\renewcommand{\left}{\mathopen{}\mathclose\bgroup\original@left}
\renewcommand{\right}{\aftergroup\egroup\original@right}

\theoremstyle{plain}
\newtheorem{theorem}{Theorem}[section]
\AfterEndEnvironment{theorem}{\@aftertheorem}
\newtheorem{lemma}[theorem]{Lemma}
\AfterEndEnvironment{lemma}{\@aftertheorem}

\AfterEndEnvironment{proposition}{\@aftertheorem}
\newtheorem{corollary}[theorem]{Corollary}
\AfterEndEnvironment{corollary}{\@aftertheorem}
\newtheorem{conj}[theorem]{Conjecture}
\AfterEndEnvironment{conj}{\@aftertheorem}

\theoremstyle{definition}

\AfterEndEnvironment{definition}{\@aftertheorem}
\newtheorem{remark}[theorem]{Remark}
\AfterEndEnvironment{remark}{\@aftertheorem}
\newtheorem{question}[theorem]{Question}
\AfterEndEnvironment{question}{\@aftertheorem}

\AfterEndEnvironment{op}{\@aftertheorem}

\AfterEndEnvironment{proof}{\@aftertheorem}

\numberwithin{equation}{section}

\makeatother

\begin{document}
\title[Arcsine laws for random walks]{Arcsine laws for random walks generated from random permutations with applications to genomics}

\author[]{{Xiao} Fang}
\address{Department of Statistics, The Chinese University of Hong Kong.
} \email{xfang@sta.cuhk.edu.hk}

\author[]{{Han Liang} Gan}
\address{Department of Mathematics, Northwestern University. 
} \email{ganhl@math.northwestern.edu}

\author[]{{Susan} Holmes}
\address{Department of Statistics, Stanford University.
} \email{susan@stat.stanford.edu  }

\author[]{{Haiyan} Huang}
\address{Department of Statistics, University of California, Berkeley.
} \email{hhuang@stat.berkeley.edu}

\author[]{{Erol} Pek\"{o}z}
\address{Questrom School of Business, Boston University.  
} \email{pekoz@bu.edu}

\author[]{{Adrian} R\"{o}llin}
\address{Department of Statistics and Applied Probability, National University of Singapore.  
} \email{adrian.roellin@nus.edu.sg}

\author[]{{Wenpin} Tang}
\address{Department of Industrial Engineering and Operations Research, UC Berkeley. 
} \email{wenpintang@stat.berkeley.edu}


\begin{abstract}
A classical result for the simple symmetric random walk with $2n$ steps is that the number of steps above the origin, the time of the last visit to the origin, and the time of the maximum height  all have exactly the same distribution and converge when scaled to the arcsine law.   Motivated by applications in genomics, we study the distributions of these statistics for the non-Markovian random walk generated from the ascents and descents of a uniform random permutation and a Mallows($q$) permutation and show that they have the same asymptotic distributions as for the simple random walk.  We also give an unexpected conjecture, along with numerical evidence and a partial proof in special cases, for the result that the number of steps above the origin by step $2n$ for the uniform permutation generated walk has exactly the same discrete arcsine distribution as for the simple random walk, even though the other statistics for these walks have very different laws.  We also give explicit error bounds to the limit theorems using Stein's method for the arcsine distribution, as well as functional central limit theorems and a strong embedding of the Mallows$(q)$ permutation which is of independent interest.
\end{abstract}

\maketitle

\noindent\textit{Key words :} Arcsine distribution, Brownian motion, L\'evy statistics, limiting distribution, Mallows permutation, random walks, Stein's method, strong embedding, uniform permutation. 

\noindent\textit{AMS 2010 Mathematics Subject Classification: } 60C05, 60J65, 05A05.

\section{Introduction}

\noindent The arcsine distribution appears surprisingly in the study of random walks and Brownian motion. 
Let $B\coloneqq (B_t; \, t \ge 0)$ be one-dimensional Brownian motion starting at $0$.
Let 
\begin{itemize}[itemsep = 4 pt]
\item
$G\coloneqq \sup\{0 \leq s \leq 1: B_s = 0\}$ be the last exit time of $B$ from zero before time $1$,
\item
$G^{\max}\coloneqq \inf \left\{0 \leq s \leq 1: B_s = \max_{u \in [0,1]} B_u\right\}$ be the first time at which $B$ achieves its maximum on $[0,1]$,
\item
$\Gamma\coloneqq  \int_0^1 1_{\{B_s > 0\}} ds$ be the occupation time of $B$ above zero before time $1$.
\end{itemize}
In \cite{Levy39, Levy65}, L\'evy proved the celebrated result that $G$, $G^{\max}$ and $\Gamma$ are all arcsine distributed with density
\begin{equation}
\label{eq:arcsine}
f(x) = \frac{1}{\pi \sqrt{x(1-x)}} \quad \text{for } 0 < x < 1. 
\end{equation}
For a random walk $S_n\coloneqq \sum_{k = 1}^n X_k$ with increments $(X_k; \, k \ge 1)$ starting at $S_0 \coloneqq 0$, the counterparts of $G$, $G^{\max}$ and $\Gamma$ are given by
\begin{itemize}[itemsep = 4 pt]
\item
$G_n\coloneqq \max\{0 \le k \le n: S_k = 0\}$ the index at which the walk last hits zero before time $n$,
\item
$G^{\max}_n \coloneqq \min\{0 \le k \le n: S_k = \max_{0 \le k \le n}S_k\}$ the index at which the walk first attains its maximum value before time $n$, 
\item
$\Gamma_n\coloneqq \sum_{k = 1}^n \I[S_k > 0]$ the number of times that the walk is strictly positive up to time $n$, 
and $N_n\coloneqq \sum_{k = 1}^{n} \I[S_{k-1} \ge 0, \, S_k \ge 0]$ the number of edges which lie above zero up to time $n$. 
\end{itemize}
The discrete analog of L\'evy's arcsine law was established by Andersen \cite{Andersen}, where the limiting distribution \eqref{eq:arcsine} was computed by Erd\"os and Kac \cite{EK}, and Chung and Feller \cite{CF49}.
Feller \cite{Feller} gave the following refined treatment:
\begin{enumerate}[label=$(\roman*)$, itemsep = 4 pt]
\item
If the increments $(X_k; \, k \ge 1)$ of the walk are exchangeable with continuous distribution, then
\begin{equation*}
\Gamma_n  \stackrel{(d)}{=} G_n^{\max}.
\end{equation*}
\item
For a simple random walk with $\mathbb{P}(X_k =  \pm 1) = 1/2$, $N_{2n} \stackrel{(d)}{=} G_{2n}$
which follows the discrete arcsine law given by
\begin{equation}
\label{eq:discretearcsine}
\alpha_{2n, 2k} \coloneqq \frac{1}{2^{2n}} \binom{2k}{k} \binom{2n - 2k}{n - k} \quad \text{for } k \in \{0, \ldots, n\}.
\end{equation} 
\end{enumerate}
In the Brownian scaling limit, the above identities imply that 
$\Gamma \stackrel{(d)}{=} G^{\max} \stackrel{(d)}{=} G$. 
The fact that $G \stackrel{(d)}{=} G^{\max}$ also follows from L\'evy's identity $(|B_t|; \, t \ge 0) \stackrel{(d)}{=} (\sup_{s \le t}B_s - B_t; \, t \ge 0)$.
See Williams \cite{Williams69}, Karatzas and Shreve \cite{KS87}, Rogers and Williams \cite[Section 53]{RW87}, Pitman and Yor \cite{PY92} for various proofs of L\'evy's arcsine law.
The arcsine law has further been generalized in several different ways, 
e.g. Dynkin \cite{Dynkin}, Getoor and Sharpe \cite{GS94}, and Bertoin and Doney \cite{BD95} to L\'evy processes;
Barlow, Pitman and Yor \cite{BPY}, and Bingham and Doney \cite{BD88} to multidimensional Brownian motion;
Akahori \cite{Aka95} and Tak{\'a}cs \cite{Tak96} to Brownian motion with drift;
Watanabe \cite{Wata95} and Kasahara and Yano \cite{KY05} to one-dimensional diffusions.
See also Pitman \cite{Pitman18} for a survey of arcsine laws arising from random discrete structures.

In this paper we are concerned with the limiting distribution of the {\em L\'evy statistics} $G_n$, $G_n^{\max}$, $\Gamma_n$ and $N_n$ of a random walk generated from a class of random permutations. Our motivation comes from a statistical problem in genomics.

\subsection{Motivation from genomics}
Understanding the relationship between genes is an important goal of systems biology. Systematically measuring the co-expression relationships between genes requires appropriate measures of the statistical association between bivariate data. Since gene expression data routinely require normalization, rank correlations such as Spearman rank correlation have been commonly used. Compared to many other measures, although some information may be lost in the process of converting numerical values to ranks, rank correlations are usually advantageous in terms of being invariant to monotonic transformation, and also robust and less sensitive to outliers.  In genomics studies, however, these correlation-based and other kinds of global measures have a practical limitation --- they measure a stationary dependent relationship between genes across all samples. It is very likely that the patterns of gene association may change or only exist in a subset of the samples, especially when the samples are pooled from heterogeneous biological conditions. In response to this consideration, several recent efforts have considered statistics that are based on counting local patterns of gene expression ranks to take into account the potentially diverse nature of gene interactions. For instance, denoting the expression profiles for genes $X$ and $Y$ over $n$ conditions (or $n$ samples) by ${\bf x}=(x_1, \dots, x_n)$ and ${\bf y}=(y_1,\dots,y_n)$ respectively, the following statistic, denoted by $W_2$, was introduced in \cite{WWH} to consider and aggregate possible local interactions: 
$$ W_2=\sum_{1 \le i_1< \dots < i_k \le n} \big( \I \left [\phi(x_{i_1}, \dots,x_{i_k})= \phi(y_{i_1}, \dots,y_{i_k}) \right ] +  \I \left [\phi(x_{i_1}, \dots,x_{i_k})= \phi(-y_{i_1}, \dots,-y_{i_k}) \right ] \big),$$
where $\I[\cdot]$ denotes the indicator function and $\phi$ is the rank function that returns the indices of elements in a vector after they have been sorted in an increasing order (for example, $\phi(0.5,1.5,0.2)=(3,1,2)$). The statistic $W_2$ aggregates the interactions across all subsamples of size $k \le n$; indeed,
$W_2$ is equal to the total number of increasing and decreasing subsequences of length $k$ in a suitably permuted sequence. To see this, suppose $\sigma$ is a permutation that sorts the elements of ${\bf y}$ in a decreasing order. Let ${\bf z}=\sigma({\bf x}) = (z_1, \dots, z_n)$ be that permutation applied to ${\bf x}$; then $W_2$ can then be rewritten as
$$W_2=\sum_{1 \le i_1< \dots < i_k \le n} \big(\I[z_{i_1}< \dots <z_{i_k}] + \I[z_{i_1} > \dots  > z_{i_k}]\big).$$
Several variants of $W_2$ have been studied to detect different types of dependent patterns between ${\bf x}$ and ${\bf y}$ (see, for example, \cite{WWH} and \cite{WLTR}).  

One variant, for example, is to have $k=2$ and consider only increasing patterns in ${\bf z}$ to assess a negative dependent relationship between ${\bf x}$ and ${\bf y}$. Denoted by $W^*$, this variant can be simply expressed as
$W^*=\sum_{1 \le i_1 < i_2 \le n} \I[z_{i_1}  < z_{i_2}].$ If a more specific negative dependent structure is concerned, say gene $Y$ is an active repressor of gene $X$ when the expression level of gene $Y$ is above a certain value, then we would expect a negative dependent relationship between ${\bf x}$ and ${\bf y}$, but with that dependence happening only locally among some vector elements. More specifically, this situation suggests that for a condition/sample, the expression of gene $X$ is expected to be low when the expression of gene $Y$ is sufficiently high, or equivalently, this dependence presents between a pair of elements (with each from ${\bf x}$ and ${\bf y}$ respectively) only when the associated element in ${\bf y}$ is above a certain value.   
To detect this type of dependent relationship, naturally we may consider the following family of statistics
\begin{equation} 
  W^*_m = \sum_{i=1}^{m}  \I[z_i < z_{i+1} ],
  \qquad 1\leq m\leq n-1.
\end{equation}
Note that the elements in ${\bf y}$ are ordered in a decreasing order. Thus in this situation that gene $Y$ is an active repressor of gene $X$ when the expression of gene $Y$ is above certain level, there should exist a change point $m_0$ such that $W_m^*$ is significantly high (in comparison to the null case that ${\bf x}$  and ${\bf y}$ are independent) when $m<m_0$ and the significance would become gradually weakened or disappear as $m$ grows from $m_0$ to $n$. For a mathematical convenience, considering $W^*_m$ is equivalent to consider
\begin{equation} 
  T_m = \sum_{i=1}^{m} (2 \I[z_{i+1} > z_i]  - 1 ),
  \qquad 1\leq m\leq n-1.
\end{equation}
As argued above, exploring the properties of this process-level statistic would be useful to understand a ``local" negative relationship between ${\bf x}$ and ${\bf y}$ that happens only among a subset of vector elements, as well as for detecting when such relationships would likely occur. To the best of our knowledge, the family of statistics $(T_m;\, 1\leq m \leq n-1)$ has not been theoretically studied in the literature. This statistic provides a motivation for studying the related problem of the permutation generated random walk.

\subsection{Permutation generated random walk}
Let $\pi\coloneqq(\pi_1, \ldots, \pi_{n+1})$ be a permutation of $[n+1]\coloneqq \{1, \ldots, n+1\}$. 
Let
\begin{equation*}
X_k \coloneqq \begin{cases}        
\, +1  & \text{if $\pi_k < \pi_{k+1}$,} \\
\, -1 & \text{if $\pi_k > \pi_{k+1}$,}
\end{cases}
\end{equation*}
and denote by $S_n\coloneqq \sum_{k=1}^n X_k$, $S_0\coloneqq0$, the corresponding walk generated by $\pi$. That is, the walk moves to the right at time $k$ if the permutation has a {\em rise} at position $k$, and the walk moves to the left at time $k$ if the permutation has a {\em descent} at position $k$.
An obvious candidate for $\pi$ is the uniform permutation of $[n+1]$. 
This random walk model was first studied by Oshanin and Voituriez \cite{OV04} in the physics literature, and also appeared in the study of the {\em zigzag diagrams} by Gnedin and Olshanski \cite{GO06}.

In this article, we consider a more general family of random permutations proposed by Mallows \cite{Mallows57}, which includes the uniform random permutation. For $0 \leq q \leq 1$, the one-parameter model
\begin{equation}
\label{eq:Mallows}
\mathbb{P}_q(\pi) = \frac{q^{\inv(\pi)}}{Z_{n,q}} \quad \text{for } \pi \text{ a permutation of } [n],
\end{equation}
is referred to as the {\em Mallows($q$) permutation} of $[n]$, where $\inv(\pi)\coloneqq \#\{(i,j) \in [n]: i<j \text{ and } \pi_i > \pi_j\}$ is the number of inversions of $\pi$ and where 
\begin{equation*}
Z_{n,q}\coloneqq \sum_{\pi} q^{\inv(\pi)} = \prod_{j=1}^n \sum_{i = 1}^j q^{i-1} = (1-q)^{-n} \prod_{j = 1 }^n ( 1 - q^j)
\end{equation*}
is known as the {\em $q$-factorial}.
For $q = 1$, the Mallows($1$) permutation is the uniform permutation of $[n]$.
There have been a line of works on this random permutation model; 
see, for example, Diaconis \cite{Diaconis88}, Gnedin and Olshanski \cite{GO09}, Starr \cite{Starr}, Basu and  Bhatnagar \cite{BB17}, Gladkich and Peled \cite{GP18}, and Tang \cite{Tang18}.

\begin{question}
\label{Q2}\label{Q1}
For a random walk generated from the Mallows($q$) permutation of $[n+1]$, what are the limit distributions of $G_n/n$, $G^{\max}_n/n$, $\Gamma_n/n$, or $N_n/n$?
\end{question}

For a Mallows($p$) permutation of $[n+1]$, the increments $(X_k; \, 1 \le k \le n)$ are not independent or even exchangeable. Moreover, the associated walk $(S_k; \, 0 \le k \le n)$ is not Markov, and as a result, the Andersen-Feller machine does not apply. Indeed, when $q=1$, this random walk has a tendency to change directions more often than a simple symmetric random walk, thus tends to cross the origin more frequently. 
Note that the distribution of the walk $(S_k; \, 0 \le k \le n)$ is completely determined by the up-down sequence, or equivalently, by the descent set $\mathcal{D}(\pi)\coloneqq \{k \in [n]: \pi_k > \pi_{k+1}\}$ of the permutation $\pi$.
The number of permutations given the up-down sequence can be expressed either as a determinant, or as a sum of multinomial coefficients; see MacMahon \cite[Vol I]{MacMahon}, Niven \cite{Niven}, de Bruijn \cite{DB}, Carlitz \cite{Carlitz}, Stanley \cite{Stanley76}, and Viennot \cite{Viennot}.
In particular, the number of permutations with a fixed number of descents is known as the {\em Eulerian number}.
See also Stanley \cite[Section 7.23]{Stanleybook2}, Borodin, Diaconis and Fulman \cite[Section 5]{BDF}, and Chatterjee and Diaconis \cite{CD17} for the descent theory of permutations.
None of these results give a simple expression for the limiting distributions of $G_n/n$, $G^{\max}_n/n$, $\Gamma_n/n$ and $N_n/n$ of a random walk generated from the uniform permutation.

\section{Main results}

\noindent To answer Question \ref{Q2}, we prove a functional central limit theorem for the walk generated from the Mallows($q$) permutation. Though for each $n > 0$ the associated walk $(S_k; \, 0 \le k \le n)$ is not Markov, the scaling limit is Brownian motion with drift. 
As a consequence, we derive the limiting distributions of the L\'evy statistics, which can be regarded as generalized arcsine laws.
In the sequel, let $(S_t; \, 0 \le t \le n)$ be the linear interpolation of the walk $(S_k; \, 0 \le k \le n)$.
That is,
\begin{equation*}
S_{t} = S_{j-1} + (t-j+1)(S_j - S_{j-1}) \quad \text{for } j-1 \le t \le j.
\end{equation*}
See Billingsley \cite[Chapter 2]{Bill2} for background on the weak convergence in the space $C[0,1]$. 
The result is stated as follows.

\begin{theorem}
\label{main}
Fix $0< q\leq 1$, and let $(S_k; \, 0 \le k \le n)$ be a random walk generated from the Mallows($q$) permutation of $[n+1]$. 
Let 
\begin{equation}
\label{eq:munu}
\mu\coloneqq \frac{1-q}{1+q} \quad \text{and} \quad \sigma \coloneqq  \sqrt{\frac{4q(1-q+q^2)}{(1+q)^2(1+q+q^2)}}.
\end{equation}
Then as $n \rightarrow \infty$,
\begin{equation}
\label{eq:cv}
\left(\frac{S_{nt}}{\sqrt{n}}; \, 0 \le t \le 1 \right) \stackrel{(d)}{\longrightarrow} \left(\mu t + \sigma B_t; \, 0 \le t \le 1 \right),
\end{equation}
where $\stackrel{(d)}{\longrightarrow}$ denotes the weak convergence in $C[0,1]$ equipped with the sup-norm topology.
\end{theorem}

\begin{remark}
Given the above theorem, it is a direct consequence (see Remark~\eqref{thmBD}) that by letting  $\nu = \mu / \sigma$, $G_n/ n \stackrel{(d)}{\longrightarrow} G$, $G^{\max}_n/n \stackrel{(d)}{\longrightarrow} G^{\max}$, $\Gamma_n/n \stackrel{(d)}{\longrightarrow} \Gamma$ and $N_n/n \stackrel{(d)}{\longrightarrow} \Gamma$ as $n \rightarrow \infty$ with
\begin{equation}
\label{eq:garcsine1}
\frac{\mathbb{P}(G \in du)}{du} = \frac{e^{-\frac{\nu^2}{2}}}{\pi \sqrt{u(1-u)}} + \frac{\nu^2}{2} \int_u^1 \frac{e^{-\frac{\nu^2 y}{2}}}{\pi \sqrt{u(y-u)}}dy,
\end{equation}
and
\begin{align}
\label{eq:garcsine2}
\frac{\mathbb{P}(G^{\max} \in du)}{du} & = \frac{\mathbb{P}(\Gamma \in du)}{du} \notag\\
& = \frac{1}{\pi \sqrt{u(1-u)}} e^{-\frac{\nu^2}{2}} + \sqrt{\frac{2}{\pi(1-u)}} \nu e^{-\frac{\nu^2(1-u)}{2}} \Phi(\nu \sqrt{u}) \notag\\
& \quad  - \sqrt{\frac{2}{\pi u}} \nu e^{-\frac{\nu^2 u}{2}} \Phi(-\nu \sqrt{1-u}) - 2 \nu^2 \Phi(\nu \sqrt{u}) \Phi(-\nu\sqrt{1-u}),
\end{align}
where $\Phi(x)\coloneqq \frac{1}{\sqrt{2 \pi}} \int_{-\infty}^x \exp(-y^2/2) dy$ is the cumulative distribution function of the standard normal distribution.
\end{remark}

The proof of Theorem \ref{main} will be given in Section \ref{s3}, which makes use of Gnedin-Olshanski's construction of the Mallows$(q)$ permutation. By letting $q= 1$,
we get the scaling limit of a random walk generated from the uniform permutation, which has recently been proved by Tarrago \cite[Proposition 9.1]{Tarrago} in the framework of zigzag graphs.  
For this case, we have the following corollary. 

\begin{corollary}
Let $(S_k; \, 0 \le k \le n)$ be a random walk generated from the uniform permutation of $[n+1]$. 
Then as $n \rightarrow \infty$,
\begin{equation}
\left(\frac{S_{nt}}{\sqrt{n}}; \, 0 \le t \le 1 \right) \stackrel{(d)}{\longrightarrow} \left(\frac{1}{\sqrt{3}}B_t; \, 0 \le t \le 1 \right),
\end{equation}
where $\stackrel{(d)}{\longrightarrow}$ denotes the weak convergence in $C[0,1]$ equipped with the sup-norm topology.
Consequently, as $n \rightarrow \infty$, the random variables $G_n/n$, $G^{\max}_n/n$ and $\Gamma_n/n$ converge in distribution to the arcsine law given by the density \eqref{eq:arcsine}.
\end{corollary}

Now that the limiting process has been established, we can ask the following question.

\begin{question} For a random walk generated from the Mallows($q$) permutation of $[n+1]$, find error bounds between $G_n/n$, $G^{\max}/n$, $\Gamma_n/n$, $N_n/n$ and their corresponding limits \eqref{eq:garcsine1}-\eqref{eq:garcsine2}.
\end{question}

While we cannot answer these questions directly, we were able to prove partial and related results. To state these, we need some notations. For two random variables $X$ and $Y$, we define the Wasserstein distance as
\begin{equation*}
\dw(X, Y) \coloneqq \sup_{h \in \Lip(1)} |\mathbb{E}h(X) - \mathbb{E}h(Y)|,
\end{equation*}
where $\Lip(1)\coloneqq \{h : |h(x) - h(y)| \le |x - y|\}$ is the class of Lipschitz-continuous functions with Lipschitz constant $1$. 
For $m \ge 1$, let $\BC^{m,1}$ be the class of bounded functions that have $m$ bounded and continuous derivatives  
and whose $m^{\mathrm{th}}$ derivative is Lipschitz continuous.
Let $\Vert h\Vert_{\infty}$ be the sup-norm of $g$, and if the $k^{\mathrm{th}}$ derivative of $h$ exists, let 
\begin{equation*}
|h|_k \coloneqq \left\Vert \frac{d^k h}{d x^k} \right\Vert_{\infty} \quad \text{and} \quad |h|_{k,1} \coloneqq  \sup_{x,y} \left| \frac{d^k h(x)}{d x^k} - \frac{d^k h(y)}{d y^k}  \right| \frac{1}{|x - y|}.
\end{equation*}

The following results hold true for a simple random walk. However, we have strong numerical evidence that they are also true for the permutation generated random walk; see Conjecture~\ref{conjecture} below.

\begin{theorem}
\label{thm:N2n}
Let $(S_{k}; \, 0 \le k \le 2n)$ be a simple symmetric random walk.
Then
\begin{equation}
\label{eq:N2narcsine}
\mathbb{P}(N_{2n} = 2k) = \alpha_{2k,2n} \quad \text{for } k \in \{0, \ldots , n\}.
\end{equation}
Moreover, let $Z$ be an arcsine distributed random variable; then
\begin{equation}
\label{eq:Wass}
\dw\left(\frac{N_{2n}}{2n},Z\right) \le \frac{27}{2n} + \frac{8}{n^2}.
\end{equation}
Furthermore, for any $h \in \BC^{2,1}$, 
\begin{equation}
\label{eq:BC21}
\left|\mathbb{E}h\left(\frac{N_{2n}}{2n}\right) - \mathbb{E}h(Z) \right| \leq \frac{4 |h|_2 + |h|_{2,1}}{64n} + \frac{|h|_{2,1}}{64n^2}.
\end{equation}
\end{theorem}

Identity \eqref{eq:N2narcsine} can be found in \cite{Feller}, the bound \eqref{eq:Wass} was proved by \cite{GR13}, and 
the proof of \eqref{eq:BC21} is deferred to Section \ref{s2}.

\begin{conj}\label{conjecture}
For a uniform random permutation generated random walk of length $2n+1$, the probability that there are $2k$ edges above the origin equals $\alpha_{2n,2k}$, which is the same as that of a simple random walk (see~\eqref{eq:discretearcsine}).
\end{conj}

For a walk generated from a permutation of $[n+1]$, call it a {\em positive walk} if $N_n = n$, and a {\em negative walk} if $N_n = 0$.
In \cite{BDN}, Bernardi, Duplantier and Nadeau proved that the number of positive walks $b_n$ generated from permutations of $[n]$ is $n!!\,(n-2)!!$ if $n$ is odd, and $[(n-1)!!]^2$ if $n$ is even.
Computer enumerations suggest that $c_{2k,2n+1}$, the number of walks generated from permutations of $[2n+1]$ with $2k$ edges above the origin, satisfies 
\begin{equation}
\label{eq:cform}
c_{2k, 2n+1} = \binom{2n+1}{2k} b_{2k} b_{2n-2k+1}.
\end{equation}
Note that, for the special cases $k = 0$ and $k=n$, the formula \eqref{eq:cform} agrees with the known results in \cite{BDN}. 
The formula \eqref{eq:cform} suggests a bijection between the set of walks generated from permutations of $[2n+1]$ with $2k$ positive edges and the set of pairs of positive walks generated from permutations of $[2k]$ and $[2n-2k+1]$ respectively.
A naive idea is to break the walk into positive and negative excursions, and exclude the final visit to the origin before crossing the other side of the origin in each excursion \cite{Andersen, Bertoin93}.
However, this approach does not work since not all pairs of positive walks are obtainable. For example, for $n = 3$, the pair $(1,2,3)$ and $(7,6,5,4)$ cannot be obtained. If Conjecture \ref{conjecture} holds, we get the arcsine law as the limiting distribution of $N_{2n}/2n$ with error bounds.

While we are not able to say much about $G_n$, $G^{\max}_n$ and $\Gamma_n$ with respect to a random walk generated from the uniform permutation for finite $n$, we can prove that the limiting distributions of these L\'evy statistics are still arcsine; this is a consequence of the fact that the scaled random walks converge to Brownian motion. 

Classical results of Skorokhod \cite{Skorokhod}, and Koml{\'o}s, Major and Tusn{\'a}dy \cite{KMT1, KMT2} provide strong embeddings of a random walk with independent increments into Brownian motion. 
In view of Theorem \ref{main}, it is also interesting to understand the strong embedding of a random walk generated from the Mallows($q$) permutation. We have the following result.

\begin{theorem}
\label{strongembed} Fix $0< q\leq 1$, and let $(S_k; \, 0 \le k \le n)$ be a random walk generated from the Mallows($q$) permutation of $[n+1]$.
Let $\mu$ and $\sigma$ be defined by \eqref{eq:munu}, and let
\begin{equation}
\beta\coloneqq \frac{2}{\sigma(1+q)}  \quad \text{and} \quad \eta\coloneqq \frac{2q}{1-q+q^2}.
\end{equation}
Then there exist universal constants $n_0, c_1, c_2 > 0$ such that for any $\varepsilon \in (0,1)$ and $n \ge n_0$, 
we can construct $(S_t; \, 0 \le t \le n)$ and $(B_t; \, 0 \le t \le n)$ on the same probability space such that
\begin{equation}
\mathbb{P}\left(\sup_{0 \le t \le n} \left| \frac{1}{\sigma} \left(S_t - \mu t \right)- B_t \right| > c_1 n^{\frac{1 + \varepsilon}{4}} (\log n)^{\frac{1}{2}} \beta \right) \le \frac{c_2(\beta^6 + \eta)}{\beta^2 n^{\varepsilon} \log n}.
\end{equation}
\end{theorem}

In fact, a much more general result, namely a strong embedding for  $m$-dependent random walks, will be proved in Section~\ref{s4}.

Also note that there is a substantial literature studying the relations between random permutations and Brownian motion.
Classical results were surveyed in Arratia, Barbour and Tavar\'e \cite{ABT}, and Pitman \cite{Pitmanbook}.
See also Janson \cite{Janson17}, Hoffman, Rizzolo and Slivken \cite{HRS1, HRS2}, and Bassino, Bouvel, F\'eray, Gerin and Pierrot \cite{BBF} for recent progress on the Brownian limit of pattern-avoiding permutations.

\section{Proof of Theorem~\ref{main}}
\label{s3}

\noindent In this section, we prove Theorem \ref{main}. To establish the result, we first show that the Mallows$(q)$ permutation can be constructed from \emph{one-dependent} increments, then calculate its moments and use an invariance principle. 
\subsection{Mallows$(q)$ permutations}
\label{s31}

Gnedin and Olshanski \cite{GO09} provide a nice construction of the Mallows$(q)$ permutation, which is implicit in the original work of Mallows \cite{Mallows57}. 
This representation of the Mallows$(q)$ permutation plays an important role in the proof of Theorem \ref{main}.

For $n>0$ and $0 < q < 1$, let $\mathcal{G}_{q,n}$ be a {\em truncated geometric} random variable on $[n]$ whose probability distribution is given by
\begin{equation}
\label{eq:TG}
\mathbb{P}(\mathcal{G}_{q,n} = k) = \frac{q^{k-1}(1-q)}{1- q^n} \quad \text{for } k \in [n].
\end{equation} 
Since $\mathbb{P}(\mathcal{G}_{q,n} = k)\to n^{-1}$ if $q\to1$, we can extend the definition of $\mathcal{G}_{q,n}$ to $q=1$, which is  just the uniform distribution on $[n]$.
The Mallows($q$) permutation $\pi$ of $[n]$ is constructed as follows. Let $(Y_k; \, k \in [n])$ be a sequence of independent random variables, where $Y_k$ is distributed as $\mathcal{G}_{n+1-k}$. 
Set
\begin{itemize}[itemsep = 4 pt]
\item 
$\pi_1 \coloneqq Y_1$, 
\item 
for $k \ge 2$, let $\pi_k \coloneqq \psi(Y_k)$ where $\psi$ is the increasing bijection from $[n - k +1]$ to \\
 $[n] \setminus \{ \pi_1, \pi_2, \cdots, \pi_{k-1} \}.$
\end{itemize}
That is, pick $\pi_1$ according to $\mathcal{G}_{q,n}$, and remove $\pi_1$ from $[n]$. 
Then pick $\pi_2$ as the $\mathcal{G}_{q,n-1}^{\thaa}$ smallest element of $[n] \setminus \{\pi_1\}$, and remove $\pi_2$ from $[n] \setminus \{\pi_1\}$, and so on. 
As immediate consequence of this construction, we have that for the increments $(X_k; \, k \in [n])$ of a random walk generated from the Mallows($q$) permutation of $[n+1]$, 
\begin{itemize}[itemsep = 4 pt]
\item
for each $k$, $\mathbb{P}(X_k = 1) = \mathbb{P}(\mathcal{G}_{q,n+1 -k} \leq \mathcal{G}_{q,n-k}) = 1/(1+q)$ which is independent of $k$ and $n$;
thus, $\mathbb{E}X_k = (1-q)/(1+q)$ and $\var X_k = 4q/(1+q)^2$;
\item
the sequence of increments $(X_k; \, k \in [n])$, though not independent, is {\em two-block factor} hence {\em one-dependent};
see de Valk \cite{deValk} for background.
\end{itemize}
Such construction is also used by Gnedin and Olshanski \cite{GO09} to construct a random permutation of positive integers, called the {\em infinite $q$-shuffle}. 
The latter is further extended by Pitman and Tang \cite{PT17} to {\em $p$-shifted permutations} as an instance of regenerative permutations, and used by Holroyd, Hutchcroft and Levy \cite{HHL17} to construction symmetric $k$-dependent $q$-coloring of positive integers. 

If $\pi$ is a uniform permutation of $[n]$, the central limit theorem of the number of descents $\# \mathcal{D}(\pi)$ is well known; that is, 
\begin{equation*}
\frac{1}{\sqrt{n}}\left(\# \mathcal{D}(\pi) - \frac{n}{2}\right) \stackrel{(d)}{\longrightarrow} \frac{1}{\sqrt{12}} \mathcal{N}(0,1),
\end{equation*}
where $\mathcal{N}(0,1)$ is standard normal distributed.
See Chatterjee and Diaconis \cite[Section 3]{CD17} for a survey of six different approaches to prove this fact.
The central limit theorem of the number of descents of the Mallows($q$) permutation is known and is as follows.
\begin{lemma}[Borodin, Diaconis and Fulman, Proposition 5.2 \cite{BDF}]
\label{CLTMallows}
Fix $0< q\leq 1$, let $\pi$ be the Mallows($q$) permutation of $[n]$, and let $\# \mathcal{D}(\pi)$ be the number of descents of $\pi$. 
Then
\begin{equation}
\mathbb{E} \# \mathcal{D}(\pi) = \frac{(n-1)q}{1+q} \quad \text{and} \quad \var \# \mathcal{D}(\pi) = q \frac{(1-q+q^2)n  -1 + 3q - q^2}{(1+q)^2(1+q+q^2)}.
\end{equation}
Moreover,
\begin{equation}
\frac{1}{\sqrt{n}}\left(\# \mathcal{D}(\pi) - \frac{nq}{1+q}\right) \stackrel{(d)}{\longrightarrow} \mathcal{N}\left(0, \frac{q(1-q+q^2)}{(1+q)^2(1+q+q^2)} \right).
\end{equation}
\end{lemma}
We are now ready to prove Theorem~\ref{main}.
\begin{proof}[Proof of Theorem~\ref{main}]
Since the increments of a permutation generated random walk are $1$-dependent, the functional CLT is an immediate consequence of \cite[Theorem~5.1]{Billingsley} and the moments in Lemma~\ref{CLTMallows}.
\end{proof}
\subsection{L\'evy statistics of Brownian motion with drift}
Let $B^{\mu, \sigma}_t \coloneqq \mu t + \sigma B_t$ be Brownian motion with drift $\mu$ and variance $\sigma^2$.
For $\mu = 0$, the L\'evy statistics $G$, $G^{\max}$ and $\Gamma$ are all arcsine distributed. 
The following remark gives a summary of the distributions of these L\'evy statistics of Brownian motion with drift.
\begin{remark}
\label{thmBD}
Let $G$, $G^{\max}$ and $\Gamma$ be the L\'evy statistics defined for $(B^{\mu,\sigma}_t; \, t \ge 0)$.
Then by letting $\nu\coloneqq \mu/\sigma$,
\begin{enumerate}[label=($\roman*$), itemsep = 4 pt]
\item
the distribution of $G$ is given by \eqref{eq:garcsine1},
\item
$G^{\max}$ has the same distribution as $\Gamma$, with the distribution given by \eqref{eq:garcsine2}.
\end{enumerate}
Part ($i$) can be derived by Girsanov's change of variables, see Iafrate and Orsingher \cite[Theorem 2.1]{IO18} for details.
For part ($ii$), the fact that $G^{\max} \stackrel{(d)}{=} \Gamma$ for $B^{\mu, \sigma}$ follows from a path transform of Embrechts, Rogers and Yor \cite[(1.b)]{ERY}. The density formula \eqref{eq:garcsine2} can be read from Akahori \cite[Theorem 1.1(i)]{Aka95}, see also Tak\'{a}cs \cite{Tak96}, and Doney and Yor \cite{DY98} for various proofs.
\end{remark}

\section{Proof of Theorem~\ref{thm:N2n}}

\label{s2}

\subsection{Stein's method for the arcsine distribution}
It is well known that for a simple symmetric walk, $G_{2n}$ and $N_{2n}$ are discrete arcsine distributed, thus converging to the arcsine distribution. To apply Stein's method for arcsine approximation we first need a characterising operator. 
\begin{lemma}
\label{eq:Steinarcsine}
A random variable $Z$ is arcsine distributed if and only if
\[
\mathbb{E} \left[ Z(1-Z) f'(Z) + \left( 1/2 - Z\right) f(Z) \right] = 0
\]
for all functions $f$ in a `rich enough' family of test functions.
\end{lemma}

To apply Stein's method, we proceed as follows. Let~$Z$ be an arcsine distributed random variable. Then for any $h\in \text{Lip}(1)$ or $h \in \BC^{2,1}$, assume we have a function~$f$ that solves
\begin{align}\label{eq:stneq}
x(1-x) f'(x) + \left( 1/2- x \right) f(x) = h(x) - \mathbb{E} h(Z).
\end{align}
Now, replacing $x$ by $W$ in \eqref{eq:stneq} and taking expectation, this yields an expression for~$\mathbb E h(W)-\mathbb E h(Z)$
in terms of just~$W$ and~$f$. Our goal is therefore to bound the expectation of the left hand side of~(\ref{eq:stneq}) by utilising properties of $f$. Extending the work of D{\"o}bler \cite{Do12}, Goldstein and Reinert \cite{GR13} developed Stein's method for the beta distribution (of which arcsine is special case) and gave an explicit Wasserstein bound between the discrete and the continuous arcsine distributions. We will use the framework from Gan, R\"{o}llin and Ross \cite{GRR} to calculate error bounds for the class of test functions $\BC^{2,1}$.

\subsection{Proof of Theorem \ref{thm:N2n}}
To simplify the notation, let 
$$W_n \coloneqq N_{2n}/2n$$
be the fraction of positive edges of a random walk generated from the uniform permutation. 
Let $\Delta_yf(x) \coloneqq f(x+y) - f(x)$.
We will use the following known facts for the discrete arcsine distribution. For any function $f \in \BC^{m,1}[0,1]$,
\begin{equation}
\label{eq:GRbound}
\mathbb{E} \left[ nW_n \left(1 - W_n + \frac{1}{2n} \right) \Delta_{1/n}f \left(W_n - \frac{1}{n} \right) + \left(\frac{1}{2} - W_n \right) f(W_n) \right] = 0.
\end{equation}
Moreover,
\begin{equation}
\label{eq:moments}
\mathbb{E}W_n = \frac{1}{2} \quad \text{and} \quad \mathbb{E}W_n^2  = \frac{3}{8} + \frac{1}{8n}.
\end{equation}
The identity \eqref{eq:GRbound} can be read from D\"{o}bler \cite{Do12}. 
The moments are easily derived by plugging in $f(x) = 0$ and $f(x) = x$; see Goldstein and Reinert \cite{GR13}.

\begin{proof}[Proof of Theorem \ref{thm:N2n}]
The distribution \eqref{eq:N2narcsine} of $N_{2n}$ can be found in Feller~\cite{Feller}.
The bound \eqref{eq:Wass} follows from the fact that $N_{2n}$ is discrete arcsine distributed, together with Theorem~1.2 of Goldstein and Reinert \cite{GR13}.

We prove the bound \eqref{eq:BC21} using the generator method. 
Recall the Stein equation \eqref{eq:Steinarcsine} for the arcsine distribution. 
First set $f = g'$, we are therefore required to bound the absolute value of
\begin{equation*}
 \mathbb{E}h(W_n) - \mathbb{E}h(Z)=  \mathbb{E}\left[W_n(1- W_n) g''(W_n) - \left(\frac{1}{2} - W_n \right) g'(W_n)\right].
\end{equation*}
Applying~\eqref{eq:GRbound} with $f$ being replaced by $g'$, we obtain
\begin{align*}
\mathbb{E}h(W_n) - &\mathbb{E}h(Z)= \mathbb{E} \left[ W_n(1-W_n) g''(W_n) - n W_n \left(1-W_n + \frac{1}{2n}\right)  \Delta_{1/n} g'\left(W_n - \frac{1}{n} \right)\right] \\
&= \mathbb{E}\left[W_n(1-W_n) \left(g''(W_n) - n \Delta_{1/n}g' \left(W_n - \frac{1}{n} \right) \right) - \frac{W_n}{2} \Delta_{1/n} g'\left(W_n - \frac{1}{n} \right)\right].
\end{align*}
The second term in the expectation is bounded by
\begin{equation}
\label{eq:secondterm}
\left| \mathbb{E} \left[\frac{W_n}{2} \Delta_{1/n} g'\left(W_n - \frac{1}{n} \right) \right]   \right| \le \frac{\mathbb{E}W_n}{2} \cdot \frac{|g|_2}{n} = \frac{|g|_2}{4n},
\end{equation}
and the first term is bounded by
\begin{align}
\label{eq:firstterm}
& \left| \mathbb{E} \left[ nW_n(1-W_n) \int_{W_n - \frac{1}{n}}^{W_n} g''(W_n) - g''(x) dx \right] \right|  \notag \\
& \qquad \le \left| \mathbb{E} \left[ nW_n(1-W_n) |g|_{2,1}\int_{W_n - \frac{1}{n}}^{W_n} |W_n -x | dx \right] \right| \notag\\
& \qquad = |g|_{2,1} n \mathbb E \left[W_n(1-W_n) \int_0^{\frac1n} s ds\right]= \frac{|g|_{2,1}}{16}\left(\frac{1}{n} + \frac{1}{n^2} \right),
\end{align}
where the last equality follows from \eqref{eq:moments}.
Combining \eqref{eq:secondterm}, \eqref{eq:firstterm} with Theorem~5 of Gan, R\"ollin and Ross~\cite{GRR} (for relating the bounds on derivatives $g$ with derivatives of $h$) yields the desired bound.
\end{proof}
\begin{remark}
The above bound is essentially sharp. Take $h(x) =\frac{x^2}{2}$, $\mathbb E h(W_n) - \mathbb E h(Z) = -\frac{1}{16n}$, and the above bound gives $| \mathbb E h(W_n) - \mathbb E h(Z)| \leq \frac{1}{16n} + \frac{1}{64n^2}$.
\end{remark}

\section{Proof of Theorem~\ref{strongembed}}
\label{s4}

\noindent In this section, we prove Theorem \ref{strongembed}.
To this end, we prove a general result for strong embeddings of a random walk with finitely dependent increments.

\subsection{Strong embeddings of $m$-dependent walks}
Let $(X_i; \, 1 \le i \le n)$ be a sequence of $m$-dependent random variables. That is, $(X_1, \ldots, X_j)$ are independent of $(X_{j+m+1}, \ldots, X_n)$ for each $j \in [n-m-1]$.
Let $(S_k; \, 0 \le k \le n)$ be a random walk with increments $X_i$, and $(S_t; \, 0 \le t \le n)$ be the linear interpolation of $(S_k; \, 0 \le k \le n)$.
Assume that the random variables $X_i$ are centered and scaled such that
\begin{equation*}
\mathbb{E}X_i = 0 \text{ for all } i \in [n] \quad \text{and} \quad \var(S_n) = n.
\end{equation*}
Let $(B_t; \, t \ge 0)$ be one-dimensional Brownian motion starting at $0$.
The idea of strong embedding is to couple $(S_t; \, 0 \le t \le n)$ and $(B_t; \, 0 \le t \le n)$ in such a way that 
\begin{equation}
\label{generalem}
\mathbb{P}\left(\sup_{0 \le t \le n}|S_t - B_t| > b_n \right) = p_n,
\end{equation}
for some $b_n = o(n^{\frac{1}{2}})$ and $p_n = o(1)$ as $n \rightarrow \infty$.

The study of such embeddings dates back to Skorokhod \cite{Skorokhod}.
When $X$'s are independent and identically distributed, Strassen \cite{Strassen} obtained \eqref{generalem} with $b_n = \mathcal{O}(n^{\frac{1}{4}} (\log n)^{\frac{1}{2}} (\log \log n)^{\frac{1}{4}})$. 
Cs{\"o}rg{\H{o}} and R{\'e}v{\'e}sz \cite{CR75} used a novel approach to prove that under the additional conditions $\mathbb{E}X_i^3 =0$ and $\mathbb{E}X_i^8 < \infty$, we get $b_n = \mathcal{O}(n^{\frac{1}{6} + \varepsilon})$ for any $\varepsilon > 0$.
Koml{\'o}s, Major and Tusn{\'a}dy \cite{KMT1, KMT2} further obtained $b_n = \mathcal{O}(\log n)$ under a finite moment generating function assumption.
See also \cite{BG16, Chatterjee12} for recent developments.

We use the argument of Cs{\"o}rg{\H{o}} and R{\'e}v{\'e}sz \cite{CR75} to obtain the following result for $m$-dependent random variables.

\begin{theorem}
\label{membed}
Let $(S_t; \, 0 \le t \le n)$ be the linear interpolation of partial sums of $m$-dependent random variables. 
Assume that $1 \le m \le n^{\frac{1}{2}}$, $\mathbb{E}X_i = 0$ for each $i \in [n]$, and $\var S_n = n  + \mathcal{O}(1)$.
Further assume that $|X_i| \le \beta$ for each $i \in [n]$, where $\beta > 0$ is a constant.
Let 
\begin{equation}
\eta\coloneqq \max_{k \in [n], \atop j \in \{0,\ldots,n-k\}} |\var(S_{j+k} - S_j) - k|.
\end{equation}
For any $\varepsilon \in (0,1)$, if $\eta \le n^{\varepsilon}$, then there exist universal constants $n_0, c_1, c_2 > 0$ such that
for any $n \ge n_0$, we can define $(S_t; \, 0 \le t \le n)$ and $(B_t; \, 0 \le t \le n)$ on the same probability space with
\begin{equation}
\mathbb{P}\left(\sup_{0 \le t \le n}|S_t - B_t| > c_1 n^{\frac{1 + \varepsilon}{4}} (\log n)^{\frac{1}{2}} m^{\frac{1}{2}} \beta \right) \le \frac{c_2 (m^4\beta^6 + \eta)}{m \beta^2 n^{\varepsilon} \log n}
\end{equation}
\end{theorem}

If $m$ and $\beta$ are constants and $\var(S_{j+k} - S_j)$ matches $k$ up to constant, from Theorem \ref{membed}, we get \eqref{generalem} with $b_n = \mathcal{O}(n^{\frac{1 + \varepsilon}{4}} (\log n)^{\frac{1}{2}})$ and $p_n = \mathcal{O}(1/(n^{\varepsilon} \log n))$ for any $\varepsilon \in (0,1)$.

\begin{proof}[Proof of Theorem \ref{strongembed}]
We apply Theorem \ref{membed} with $m = 1$, and a suitable choice of $\beta$ and $\eta$.
By centering and scaling, we consider the walk $(S_t^{'}; \, 0 \le t \le n)$ with increments $X_i^{'} = \frac{1}{\sigma} (X_i -\mu)$.
It is easy to see that
\begin{equation*}
|X_i^{'}| \le \frac{1}{\sigma} \max\left( 1- \mu,  1 + \mu \right) = \beta.
\end{equation*}
According to the result in Section \ref{s31}, 
\begin{equation*}
\mathbb{P}(X_k = X_{k+1} = 1) = \mathbb{P}(\mathcal{G}_{q, n+1-k} \le \mathcal{G}_{q,n-k} \le \mathcal{G}_{q,n-k-1}) = \frac{1}{(1+q)(1+q+q^2)}.
\end{equation*}
Elementary computation shows that for $k \le n$,
$\var S_k^{'} = k + \eta$,
which leads to the desired result.
\end{proof}

\subsection{Proof of Theorem \ref{membed}}
The proof of Theorem \ref{membed} boils down to a series of lemmas.
We use $C$ and $c$ to denote positive constants which may differ in expressions. 
Let 
\begin{equation}
\label{d}
d\coloneqq \lceil n^{\frac{1-\varepsilon}{2}} \rceil,
\end{equation}
where $\lceil x \rceil$ is the least integer greater than or equal to $x$.
We divide the interval $[0,n]$ into $d$ subintervals by points $\lceil jn/d \rceil, j\in [d]$, each with length 
$l=\lceil n/d \rceil$ or $l=\lceil n/d \rceil-1$. 
The following results hold for both values of $l$.

\begin{lemma}
Under the assumptions in Theorem \ref{membed}, we have
\begin{equation}
\label{mbetal}
3m\beta^2\geq 1 \quad \text{and} \quad l\geq 6m\log n
\end{equation}
for sufficiently large $n$.
\end{lemma}
\begin{proof}
Note that
\begin{equation*}
n = \var S_n = \sum_{i = 1}^n \sum_{j: |j - i| \le m} \mathbb{E}X_i X_j \le n(2m+1)\beta^2 \text{ and } m \ge 1,
\end{equation*}
which implies $3m \beta^2 \ge 1$. 
The second bound follows from the fact that $m \le n^{\frac{1}{2}}$ and $l \sim n^{\frac{1 + \varepsilon}{2}}$.
\end{proof}

Given two probability measures $\mu$ and $\nu$ on $\mathbb{R}$, define the Wasserstein-$2$ distance by
\begin{equation*}
d_{W_2}(\mu, \nu) = \left( \inf_{\pi \in \Gamma(\mu, \nu)} \int |x-y|^2 d \pi(x,y)\right)^{\frac{1}{2}},
\end{equation*}
where $\Gamma(\mu, \nu)$ is the space of all probability measures on $\mathbb{R}^2$ with $\mu$ and $\nu$ as marginals.
In the next two lemmas, $\mathcal{N}(\mu,\sigma^2)$ denotes a normal random variable with mean $\mu$ and variance $\sigma^2$.
\begin{lemma}
\label{W2bound}
Under the assumptions in Theorem \ref{membed}, we have for $n$ sufficiently large,
\begin{equation}
\label{W2}
d_{W_2}(S_{l -m}, \, \mathcal{N}(0,\sigma^2)) \le C m^2 \beta^3,
\end{equation}
where $\sigma^2\coloneqq \var S_{l - m}$.
\end{lemma}
\begin{proof}
Specializing Fang \cite[Corollary 2.3]{Fang18} to sums of $m$-dependent and bounded variables, we get
\begin{align*}
d_{W_2}(S_{l-m}, \mathcal{N}(0,\sigma^2)) & = \sigma \, d_{W_2}(\sigma^{-1} S_{l-m}, \mathcal{N}(0,1)) \\
& \le \sigma C\left( lm^2 \left(\frac{\beta}{\sigma} \right)^3 + \left(lm^3 \left(\frac{\beta}{\sigma} \right)^4 \right)^{\frac{1}{2}} \right) \le Cm^2 \beta^3,
\end{align*}
where we used $3 m \beta^2 \ge 1$ in \eqref{mbetal}, and $\sigma^2 \ge l-m-\eta \ge cl$ for sufficiently large $n$ from $m \le n^{\frac{1}{2}}$, $\eta \le n^{\varepsilon}$, $l \sim n^{\frac{1 + \varepsilon}{2}}$ and $\varepsilon \in (0,1)$.
\end{proof}

\begin{lemma}
\label{ecoupling}
There exists a coupling of $(S_t; \, 0 \le t \le n)$ and $(B_t; \, 0 \le t \le n)$ such that with
\begin{equation*}
e_j\coloneqq(S_{\lceil jn/d \rceil}-S_{\lceil (j-1)n/d \rceil})-(B_{\lceil jn/d \rceil}-B_{\lceil (j-1)n/d \rceil}),
\end{equation*}
the sequence $(e_1, \ldots, e_d)$ are $1$-dependent, and 
\begin{equation*}
\mathbb{E}e_j^2 \le C(m^4 \beta^6 + \eta), \quad \text{for all } j \in [n].
\end{equation*}
\end{lemma}
\begin{proof}
We use $3 m \beta^2 \ge 1$ below implicitly to absorb a few terms into $Cm^4 \beta^{6}$.
With $\sigma^2$ defined in Lemma \ref{W2bound}, we have
\begin{equation*}
d_{W_2}(\mathcal{N}(0,\sigma^2), \, \mathcal{N}(0,l)) \le \sqrt{|l- \sigma^2|} \le \sqrt{m + \eta}.
\end{equation*}
Combining \eqref{W2}, the above bound and the $m$-dependence, we can couple 
$S_{\lceil jn/d \rceil-m}-S_{\lceil (j-1)n/d \rceil}$ and $B_{\lceil jn/d \rceil}-B_{\lceil (j-1)n/d \rceil}$ for each $j\in [d]$ independently with
\begin{equation*}
\mathbb{E}[(S_{\lceil jn/d \rceil-m}-S_{\lceil (j-1)n/d \rceil})-(B_{\lceil jn/d \rceil}-B_{\lceil (j-1)n/d \rceil})]^2\leq C( m^4 \beta^6 + \eta).
\end{equation*}
By the $m$-dependence assumption, we can generate $X_1,\dots, X_n$ from their conditional distribution given $(S_{\lceil jn/d \rceil-m}-S_{\lceil (j-1)n/d \rceil}; \, j\in [d])$, thus obtaining $(S_{t}; \, 0 \le t \le n)$, and generate $(B_t; 0 \le t \le n)$ given $(B_{\lceil jn/d \rceil}; \, j\in [d])$.
Since 
\begin{equation*}
\mathbb{E}(S_{\lceil jn/d \rceil}-S_{\lceil jn/d \rceil-m})^2\leq Cm^2\beta^2,
\end{equation*}
we have
\begin{equation*}
\mathbb{E}(e_j^2)\leq C(m^4\beta^6+\eta)
\end{equation*}
Finally, the $1$-dependence of $(e_1, \ldots, e_d)$ follows from the $m$-dependence assumption.
\end{proof}

\begin{lemma}
\label{Tbounds}
Let $T_j=\sum_{i=1}^d e_i, \,  j\in [d]$. For each $b>0$, we have
\begin{equation*}
\mathbb{P}\left(\max_{j\in [d]}|T_j|>b \right)\leq C(m^4\beta^6+\eta ) d/b^2.
\end{equation*}
\end{lemma}
\begin{proof}
Define
\begin{equation*}
T_j^{(1)}=\sum_{i=1,3,5,\dots\atop i\leq j} e_i,\quad T_j^{(2)}=\sum_{i=2,4,6,\dots\atop i\leq j} e_i.
\end{equation*}
By Lemma \ref{ecoupling}, $T_j^{(1)}$ is a sum of independent random variables with zero mean and finite second moments.
By Kolmogorov's maximal inequality,
\begin{equation*}
\mathbb{P}\left(\max_{1\leq j\leq d}|T_j^{(1)}|>\frac{b}{2} \right) \leq \frac{C(m^4 \beta^6+\eta)d}{b^2}.
\end{equation*}
The same bound holds for $T_j^{(2)}$. 
The lemma is proved by observing that
\begin{equation*}
\mathbb{P}\left(\max_{j\in [d]}|T_j|>b\right)\leq \mathbb{P}\left(\max_{1\leq j\leq d}|T_j^{(1)}|>\frac{b}{2}\right)+\mathbb{P}\left(\max_{1\leq j\leq d}|T_j^{(2)}|>\frac{b}{2}\right).
\end{equation*}
\end{proof}

\begin{lemma}
\label{Sbounds}
For any $0<b\leq 4l \beta$, we have
\begin{equation*}
\mathbb{P}\left(\max_{j\in [l]} |S_j- j S_l/l|>b\right)\leq 2l \exp\left(-\frac{b^2}{48 l m \beta^2}\right).
\end{equation*}
\end{lemma}
\begin{proof}
We first prove a concentration inequality for $S_j$, then use the union bound.
Let $h(\theta)=\mathbb{E} e^{\theta S_j}$, with $h(0)=1$.
Let $S_j^{(i)}=S_j-\sum_{k\in [j]: |k-i|\leq m} X_k$. 
Using $\mathbb{E} X_i=0$, $|X_i|\leq \beta$, the $m$-dependence and the inequality
\begin{equation*}
\left|  \frac{e^x-e^y}{x-y}  \right|\leq \frac{1}{2}(e^x+e^y),
\end{equation*}
we have for $\theta > 0$,
\begin{align*}
h'(\theta)=\mathbb{E} (S_j e^{\theta S_j}) & =\sum_{i=1}^j \mathbb{E} X_i (e^{\theta S_j}-e^{\theta S_j^{(i)}}) \\
& \leq  \frac{\theta}{2} \sum_{i=1}^j \mathbb{E} |X_i| |S_j-S_j^{(i)}| (e^{\theta S_j}+e^{\theta S_j^{(i)}} )\\
&\leq  \left(m+\frac{1}{2}\right) \theta l  \beta^2 \mathbb{E} e^{\theta S_j} (1+e^{\theta (2m+1)\beta}) 
\leq  6 \theta l m \beta^2 h(\theta),
\end{align*}
for $\theta (2m+1)\beta\leq 1$.
This implies that $\log h(\theta)\leq 3l m \beta^2 \theta^2$, and 
\begin{equation*}
\mathbb{P}(S_j>b/2)\leq e^{-\theta b/2} \mathbb{E} e^{\theta S_j}\leq \exp\left(-\frac{b^2}{48 l m \beta^2}\right),
\end{equation*}
by choosing $\theta=b/(12lm\beta^2)$ provided that $b\leq 4l\beta$.
The same bound holds for $-S_j$.
Consequently,
\begin{align*}
\mathbb{P}\left(\max_{j\in [l]} |S_j- j S_l/l|>b \right) & \leq  \mathbb{P}\left(\max_{j\in [l-1]} |S_j|>b/2\right)+\mathbb{P}(|S_l|\geq b/2)\\
& \leq  2l \exp\left(-\frac{b^2}{48 l m \beta^2}\right).
\end{align*}
\end{proof}

\begin{lemma}
\label{Bbounds}
For each $b>0$, we have
\begin{equation*}
\mathbb{P}\left(\sup_{0\leq t\leq l} |B_t-t B_l/l|>b\right)\leq 2 e^{-\frac{b^2}{2l}}.
\end{equation*}
\end{lemma}
\begin{proof}
We have, by symmetry, a reflection argument for Brownian bridge and a normal tail bound,
\begin{align*}
\mathbb{P}\left(\sup_{0\leq t\leq l} |B_t-t B_l/l|>b \right) & \leq  2 \mathbb{P} \left(\sup_{0\leq t\leq l} (B_t-t B_l/l)>b\right) \\
& \leq  4 \mathbb{P}(B_l>b) \leq  2 e^{-\frac{b^2}{2l}}.
\end{align*}
\end{proof}

Now we proceed to proving Theorem \ref{membed}.
\begin{proof}[Proof of Theorem \ref{membed}]
Let 
\begin{equation*}
b=(96lm\beta^2 \log n)^{1/2}.
\end{equation*}
It satisfies $b\leq 4l\beta$ in Lemma \ref{Sbounds} since $m\leq l/(6 \log n)$ by \eqref{mbetal}. 
Note that if $\sup_{0\leq t\leq n}|S_t-B_t|>3b$, then either $\max_{j\in [d]}|T_j|>b$, or the fluctuation of either $S_t$ or $B_t$ within each subinterval is larger than $b$. 
By the union bound and Lemmas \ref{Tbounds}--\ref{Bbounds}, we have
\begin{equation*}
\mathbb{P}\left(\sup_{0\leq t\leq n}|S_t-B_t|>3b\right)\leq \frac{C(m^4\beta^6+\eta)}{m\beta^2 n^{\varepsilon}\log n}.
\end{equation*}
This proves the theorem.
\end{proof}

 \bigskip
\section*{Acknowledgments} 
\noindent The initial portion of this work was conducted at the meeting \emph{Stein's method and applications in high-dimensional statistics} held at the American Institute of Mathematics in August 2018. We are indebted to Bhaswar Bhattacharya, Sourav Chatterjee, Persi Diaconis and Jon Wellner for helpful discussions throughout the project. We would also like to express our gratitude to John Fry and staff Estelle Basor, Brian Conrey, and Harpreet Kaur at the American Institute of Mathematics for  the generosity and excellent hospitality in hosting this meeting at the Fry's Electronics corporate headquarters in San Jose, CA, and Jay Bartroff, Larry Goldstein, Stanislav Minsker and Gesine Reinert for organizing such a stimulating meeting. 

WT thanks Yuting Ye for communicating the problem of the limiting distribution of the sojourn time of a random walk generated from the uniform permutation, which brings him to this work, and Jim Pitman for helpful discussions. SH acknowledges support from the NSF DMS grant 1501767. XF acknowledges support from Hong Kong RGC ECS 24301617, GRF 14302418, 14304917.

Finally, we thank the Institute for Mathematical Sciences at the National University of Singapore, where part of this work was continued, for their kind support.


\begin{thebibliography}{68}
\providecommand{\natexlab}[1]{#1}
\providecommand{\url}[1]{\texttt{#1}}
\expandafter\ifx\csname urlstyle\endcsname\relax
  \providecommand{\doi}[1]{doi: #1}\else
  \providecommand{\doi}{doi: \begingroup \urlstyle{rm}\Url}\fi

\bibitem[Akahori(1995)]{Aka95}
J.~Akahori (1995).
\newblock Some formulae for a new type of path-dependent option.
\newblock \emph{Ann. Appl. Probab.} \textbf{5}, \penalty0 383--388.

\bibitem[Andersen(1953)]{Andersen}
E.~S. Andersen (1953).
\newblock On sums of symmetrically dependent random variables.
\newblock \emph{Skand. Aktuarietidskr.} \textbf{36}, \penalty0 123--138.

\bibitem[Arratia et~al.(2003)Arratia, Barbour and Tavar{\'e}]{ABT}
R.~Arratia, A.~Barbour and S.~Tavar{\'e} (2003).
\newblock \emph{Logarithmic combinatorial structures: a probabilistic
  approach}, volume~1.
\newblock European Mathematical Society.

\bibitem[Barlow et~al.(1989)Barlow, Pitman and Yor]{BPY}
M.~Barlow, J.~Pitman and M.~Yor (1989).
\newblock Une extension multidimensionnelle de la loi de l'arc sinus.
\newblock In \emph{S{\'e}minaire de Probabilit{\'e}s XXIII}, pages 294--314.

\bibitem[Bassino et~al.(2018)Bassino, Bouvel, F{\'e}ray, Gerin and
  Pierrot]{BBF}
F.~Bassino, M.~Bouvel, V.~F{\'e}ray, L.~Gerin and A.~Pierrot (2018).
\newblock The {B}rownian limit of separable permutations.
\newblock \emph{Ann. Probab.} \textbf{46}, \penalty0 2134--2189.

\bibitem[Basu and Bhatnagar(2017)]{BB17}
R.~Basu and N.~Bhatnagar (2017).
\newblock Limit theorems for longest monotone subsequences in random {M}allows
  permutations.
\newblock \emph{Ann. Inst. H. Poincar\'e Probab. Statist.} \textbf{53},
  \penalty0 1934--1951.

\bibitem[Bernardi et~al.(2010)Bernardi, Duplantier and Nadeau]{BDN}
O.~Bernardi, B.~Duplantier and P.~Nadeau (2010).
\newblock A bijection between well-labelled positive paths and matchings.
\newblock \emph{S{\'e}minaire Lotharingien de Combinatoire} \textbf{63},
  \penalty0 B63e.

\bibitem[Bertoin and Doney(1997)]{BD95}
J.~Bertoin and R.~Doney (1997).
\newblock Spitzer's condition for random walks and {L}{\'e}vy processes.
\newblock \emph{Ann. Inst. H. Poincar\'e Probab. Statist.} \textbf{33},
  \penalty0 167--178.

\bibitem[Bertoin(1993)]{Bertoin93}
J.~Bertoin (1993).
\newblock Splitting at the infimum and excursions in half-lines for random
  walks and {L}{\'e}vy processes.
\newblock \emph{Stochastic processes and their applications} \textbf{47},
  \penalty0 17--35.

\bibitem[Bhattacharjee and Goldstein(2016)]{BG16}
C.~Bhattacharjee and L.~Goldstein (2016).
\newblock On strong embeddings by {S}tein's method.
\newblock \emph{Electronic Journal of Probability} \textbf{21}.

\bibitem[Billingsley(1999)]{Bill2}
P.~Billingsley (1999).
\newblock \emph{Convergence of probability measures}.
\newblock Wiley Series in Probability and Statistics: Probability and
  Statistics. John Wiley \& Sons Inc., New York, second edition.

\bibitem[Billingsley(1956)]{Billingsley}
P.~Billingsley (1956).
\newblock The invariance principle for dependent random variables.
\newblock \emph{Transactions of the American Mathematical Society} \textbf{83},
  \penalty0 250--268.

\bibitem[Bingham and Doney(1988)]{BD88}
N.~Bingham and R.~Doney (1988).
\newblock On higher-dimensional analogues of the arc-sine law.
\newblock \emph{Journal of Applied Probability} \textbf{25}, \penalty0
  120--131.

\bibitem[Borodin et~al.(2010)Borodin, Diaconis and Fulman]{BDF}
A.~Borodin, P.~Diaconis and J.~Fulman (2010).
\newblock On adding a list of numbers (and other one-dependent determinantal
  processes).
\newblock \emph{Bull. Amer. Math. Soc. (N.S.)} \textbf{47}, \penalty0 639--670.


\bibitem[Carlitz(1973)]{Carlitz}
L.~Carlitz (1973).
\newblock Permutations with prescribed pattern.
\newblock \emph{Math. Nachr.} \textbf{58}, \penalty0 31--53.

\bibitem[Chatterjee(2012)]{Chatterjee12}
S.~Chatterjee (2012).
\newblock A new approach to strong embeddings.
\newblock \emph{Probability Theory and Related Fields} \textbf{152}, \penalty0
  231--264.

\bibitem[Chatterjee and Diaconis(2017)]{CD17}
S.~Chatterjee and P.~Diaconis (2017).
\newblock A central limit theorem for a new statistic on permutations.
\newblock \emph{Indian Journal of Pure and Applied Mathematics} \textbf{48},
  \penalty0 561--573.

\bibitem[Chung and Feller(1949)]{CF49}
K.~L. Chung and W.~Feller (1949).
\newblock On fluctuations in coin-tossing.
\newblock \emph{Proceedings of the National Academy of Sciences}pages 605--608.

\bibitem[Cs{\"o}rg{\H{o}} and R{\'e}v{\'e}sz(1975)]{CR75}
M.~Cs{\"o}rg{\H{o}} and P.~R{\'e}v{\'e}sz (1975).
\newblock A new method to prove strassen type laws of invariance principle.
  {I}.
\newblock \emph{Zeitschrift f{\"u}r Wahrscheinlichkeitstheorie und verwandte
  Gebiete} \textbf{31}, \penalty0 255--259.

\bibitem[de~Bruijn(1970)]{DB}
N.~G. de~Bruijn (1970).
\newblock Permutations with given ups and downs.
\newblock \emph{Nieuw Arch. Wisk. (3)} \textbf{18}, \penalty0 61--65.
\newblock ISSN 0028-9825.

\bibitem[de~Valk(1994)]{deValk}
V.~de~Valk (1994).
\newblock \emph{One-dependent {P}rocesses: {T}wo-block-factors and
  {N}on-two-block-factors}.
\newblock CWI Tracts.

\bibitem[Diaconis(1988)]{Diaconis88}
P.~Diaconis (1988).
\newblock \emph{Group representations in probability and statistics},
  volume~11.
\newblock Lecture Notes-Monograph Series.

\bibitem[D{\"o}bler(2012)]{Do12}
C.~D{\"o}bler (2012).
\newblock A rate of convergence for the arcsine law by {S}tein's method.
\newblock \emph{arXiv:1207.2401}.

\bibitem[Doney and Yor(1998)]{DY98}
R.~Doney and M.~Yor (1998).
\newblock On a formula of {T}ak{\'a}cs for {B}rownian motion with drift.
\newblock \emph{Journal of Applied Probability} \textbf{35}, \penalty0
  272--280.

\bibitem[Dynkin(1965)]{Dynkin}
E.~B. Dynkin (1965).
\newblock \emph{Markov processes. {V}ols. {I}, {II}}, volume 122 of \emph{Die
  Grundlehren der Mathematischen Wissenschaften, B\"ande 121}.
\newblock Springer-Verlag, Berlin-G\"ottingen-Heidelberg.

\bibitem[Embrechts et~al.(1995)Embrechts, Rogers and Yor]{ERY}
P.~Embrechts, L.~C.~G. Rogers and M.~Yor (1995).
\newblock A proof of {D}assios' representation of the {$\alpha$}-quantile of
  {B}rownian motion with drift.
\newblock \emph{Ann. Appl. Probab.} \textbf{5}, \penalty0 757--767.


\bibitem[Erd{\"o}s and Kac(1947)]{EK}
P.~Erd{\"o}s and M.~Kac (1947).
\newblock On the number of positive sums of independent random variables.
\newblock \emph{Bull. Amer. Math. Soc.} \textbf{53}, \penalty0 1011--1020.
\newblock ISSN 0002-9904.

\bibitem[Fang(2018)]{Fang18}
X.~Fang (2018).
\newblock Wasserstein-$2$ bounds in normal approximation under local
  dependence.
\newblock \emph{arXiv:1807.05741}.

\bibitem[Feller(1968)]{Feller}
W.~Feller (1968).
\newblock \emph{An {I}ntroduction to {P}robability {T}heory and {I}ts
  {A}pplications. {V}ol. {I}.}
\newblock Second edition. John Wiley \& Sons, Inc., New York-London-Sydney.

\bibitem[Gan et~al.(2017)Gan, R{\"o}llin and Ross]{GRR}
H.~L. Gan, A.~R{\"o}llin and N.~Ross (2017).
\newblock Dirichlet approximation of equilibrium distributions in {C}annings
  models with mutation.
\newblock \emph{Advances in Applied Probability} \textbf{49}, \penalty0
  927--959.

\bibitem[Getoor and Sharpe(1994)]{GS94}
R.~Getoor and M.~Sharpe (1994).
\newblock On the arc-sine laws for {L}{\'e}vy processes.
\newblock \emph{Journal of Applied Probability} \textbf{31}, \penalty0 76--89.

\bibitem[Gladkich and Peled(2018)]{GP18}
A.~Gladkich and R.~Peled (2018).
\newblock On the cycle structure of {M}allows permutations.
\newblock \emph{Ann. Probab.} \textbf{46}, \penalty0 1114--1169.

\bibitem[Gnedin and Olshanski(2006)]{GO06}
A.~Gnedin and G.~Olshanski (2006).
\newblock Coherent permutations with descent statistic and the boundary problem
  for the graph of zigzag diagrams.
\newblock \emph{Int. Math. Res. Not.} Art. ID 51968, 39.

\bibitem[Gnedin and Olshanski(2010)]{GO09}
A.~Gnedin and G.~Olshanski (2010).
\newblock {$q$}-exchangeability via quasi-invariance.
\newblock \emph{Ann. Probab.} \textbf{38}, \penalty0 2103--2135.

\bibitem[Goldstein and Reinert(2013)]{GR13}
L.~Goldstein and G.~Reinert (2013).
\newblock Stein's method for the beta distribution and the
  {P}olya-{E}ggenberger urn.
\newblock \emph{Journal of Applied Probability} \textbf{50}, \penalty0
  1187--1205.

\bibitem[Hoffman et~al.(2017{\natexlab{a}})Hoffman, Rizzolo and Slivken]{HRS1}
C.~Hoffman, D.~Rizzolo and E.~Slivken (2017{\natexlab{a}}).
\newblock Pattern-avoiding permutations and {B}rownian excursion part {I}:
  {S}hapes and fluctuations.
\newblock \emph{Random Structures \& Algorithms} \textbf{50}, \penalty0
  394--419.

\bibitem[Hoffman et~al.(2017{\natexlab{b}})Hoffman, Rizzolo and Slivken]{HRS2}
C.~Hoffman, D.~Rizzolo and E.~Slivken (2017{\natexlab{b}}).
\newblock Pattern-avoiding permutations and {B}rownian excursion, part {II}:
  fixed points.
\newblock \emph{Probability Theory and Related Fields} \textbf{169}, \penalty0
  377--424.

\bibitem[Holroyd et~al.(2017)Holroyd, Hutchcroft and Levy]{HHL17}
A.~Holroyd, T.~Hutchcroft and A.~Levy (2017).
\newblock Mallows permutations and finite dependence.
\newblock \emph{arXiv:1706.09526}.

\bibitem[Iafrate and Orsingher(2017)]{IO18}
F.~Iafrate and E.~Orsingher (2017).
\newblock The last zero crossing of an iterated {B}rownian motion with drift.
\newblock \emph{arXiv:1803.00877}.

\bibitem[Janson(2017)]{Janson17}
S.~Janson (2017).
\newblock Patterns in random permutations avoiding the pattern 132.
\newblock \emph{Combinatorics, Probability and Computing} \textbf{26},
  \penalty0 24--51.

\bibitem[Karatzas and Shreve(1987)]{KS87}
I.~Karatzas and S.~E. Shreve (1987).
\newblock A decomposition of the {B}rownian path.
\newblock \emph{Statist. Probab. Lett.} \textbf{5}, \penalty0 87--93.

\bibitem[Kasahara and Yano(2005)]{KY05}
Y.~Kasahara and Y.~Yano (2005).
\newblock On a generalized arc-sine law for one-dimensional diffusion
  processes.
\newblock \emph{Osaka Journal of Mathematics} \textbf{42}, \penalty0 1--10.

\bibitem[Koml{\'o}s et~al.(1975)Koml{\'o}s, Major and Tusn{\'a}dy]{KMT1}
J.~Koml{\'o}s, P.~Major and G.~Tusn{\'a}dy (1975).
\newblock An approximation of partial sums of independent {RV}'-s, and the
  sample {DF}. {I}.
\newblock \emph{Zeitschrift f{\"u}r Wahrscheinlichkeitstheorie und verwandte
  Gebiete} \textbf{32}, \penalty0 111--131.

\bibitem[Koml{\'o}s et~al.(1976)Koml{\'o}s, Major and Tusn{\'a}dy]{KMT2}
J.~Koml{\'o}s, P.~Major and G.~Tusn{\'a}dy (1976).
\newblock An approximation of partial sums of independent {RV}'s, and the
  sample {DF}. {II}.
\newblock \emph{Zeitschrift f{\"u}r Wahrscheinlichkeitstheorie und verwandte
  Gebiete} \textbf{34}, \penalty0 33--58.

\bibitem[L{\'e}vy(1939)]{Levy39}
P.~L{\'e}vy (1939).
\newblock Sur certains processus stochastiques homog\`enes.
\newblock \emph{Compositio Math.} \textbf{7}, \penalty0 283--339.
\newblock ISSN 0010-437X.

\bibitem[L{\'e}vy(1965)]{Levy65}
P.~L{\'e}vy (1965).
\newblock \emph{Processus stochastiques et mouvement brownien}.
\newblock Suivi d'une note de M. Lo\`eve. Deuxi\`eme \'edition revue et
  augment\'ee. Gauthier-Villars \& Cie, Paris.

\bibitem[MacMahon(1960)]{MacMahon}
P.~A. MacMahon (1960).
\newblock \emph{Combinatory analysis}.
\newblock Two volumes (bound as one). Chelsea Publishing Co., New York.

\bibitem[Mallows(1957)]{Mallows57}
C.~L. Mallows (1957).
\newblock Non-null ranking models. {I}.
\newblock \emph{Biometrika} \textbf{44}, \penalty0 114--130.

\bibitem[Niven(1968)]{Niven}
I.~Niven (1968).
\newblock A combinatorial problem of finite sequences.
\newblock \emph{Nieuw Arch. Wisk} \textbf{16}, \penalty0 116--123.

\bibitem[Oshanin and Voituriez(2004)]{OV04}
G.~Oshanin and R.~Voituriez (2004).
\newblock Random walk generated by random permutations of $\{$1, 2, 3, \ldots,
  n+1$\}$.
\newblock \emph{Journal of Physics A: Mathematical and General} \textbf{37},
  \penalty0 6221.

\bibitem[Pitman(2018)]{Pitman18}
J.~Pitman (2018).
\newblock Random weighted averages, partition structures and generalized
  arcsine laws.
\newblock \emph{arXiv:1804.07896.}

\bibitem[Pitman(2006)]{Pitmanbook}
J.~Pitman (2006).
\newblock \emph{Combinatorial stochastic processes}, volume 1875 of
  \emph{Lecture Notes in Mathematics}.
\newblock Springer-Verlag, Berlin.
\newblock ISBN 978-3-540-30990-1; 3-540-30990-X.

\bibitem[Pitman and Tang(2019)]{PT17}
J.~Pitman and W.~Tang (2019).
\newblock Regenerative random permutations of integers.
\newblock \emph{Ann. Probab.} \textbf{47}, \penalty0 1378--1416.

\bibitem[Pitman and Yor(1992)]{PY92}
J.~Pitman and M.~Yor (1992).
\newblock Arcsine laws and interval partitions derived from a stable
  subordinator.
\newblock \emph{Proc. London Math. Soc. (3)} \textbf{65}, \penalty0 326--356.

\bibitem[Rogers and Williams(1987)]{RW87}
L.~C.~G. Rogers and D.~Williams (1987).
\newblock \emph{Diffusions, {M}arkov processes, and martingales. {V}ol. 2}.
\newblock Wiley Series in Probability and Mathematical Statistics. John Wiley
  \& Sons, Inc., New York.
\newblock ISBN 0-471-91482-7.

\bibitem[Skorokhod(1965)]{Skorokhod}
A.~V. Skorokhod (1965).
\newblock \emph{Studies in the theory of random processes}.
\newblock Translated from the Russian by Scripta Technica, Inc. Addison-Wesley
  Publishing Co., Inc., Reading, Mass.

\bibitem[Stanley(1976)]{Stanley76}
R.~Stanley (1976).
\newblock Binomial posets, m{\"o}bius inversion, and permutation enumeration.
\newblock \emph{Journal of Combinatorial Theory, Series A} \textbf{20},
  \penalty0 336--356.

\bibitem[Stanley(1999)]{Stanleybook2}
R.~Stanley (1999).
\newblock \emph{Enumerative combinatorics. {V}ol. 2}, volume~62 of
  \emph{Cambridge Studies in Advanced Mathematics}.
\newblock Cambridge University Press, Cambridge.

\bibitem[Starr(2009)]{Starr}
S.~Starr (2009).
\newblock Thermodynamic limit for the {M}allows model on ${S}_n$.
\newblock \emph{Journal of mathematical physics} \textbf{50}, \penalty0 095208.

\bibitem[Strassen(1967)]{Strassen}
V.~Strassen (1967).
\newblock Almost sure behavior of sums of independent random variables and
  martingales.
\newblock In \emph{Proceedings of the Fifth Berkeley Symposium on Mathematical
  Statistics and Probability, volume 2}.

\bibitem[Tak{\'a}cs(1996)]{Tak96}
L.~Tak{\'a}cs (1996).
\newblock On a generalization of the arc-sine law.
\newblock \emph{Ann. Appl. Probab.} \textbf{6}, \penalty0 1035--1040.

\bibitem[Tang(2019)]{Tang18}
W.~Tang (2019).
\newblock Mallows ranking models: maximum likelihood estimate and regeneration.
\newblock In \emph{Proceedings of the 36th International Conference on Machine
  Learning}, volume~97 of \emph{Proceedings of Machine Learning Research},
  pages 6125--6134.

\bibitem[Tarrago(2018)]{Tarrago}
P.~Tarrago (2018).
\newblock Zigzag diagrams and {M}artin boundary.
\newblock \emph{Ann. Probab.} \textbf{46}, \penalty0 2562--2620.

\bibitem[Viennot(1979)]{Viennot}
G.~Viennot (1979).
\newblock Permutations ayant une forme donn\'{e}e.
\newblock \emph{Discrete Math.} \textbf{26}, \penalty0 279--284.

\bibitem[Wang et~al.(2014)Wang, Waterman and Huang]{WWH}
R.~Wang, M.~Waterman and H.~Huang (2014).
\newblock Gene coexpression measures in large heterogeneous samples using count
  statistics.
\newblock \emph{Proceedings of the National Academy of Sciences} \textbf{111},
  \penalty0 16371--16376.

\bibitem[Wang et~al.(2017)Wang, Liu, Theusch, Rotter, Medina, Waterman and
  Huang]{WLTR}
R.~Wang, K.~Liu, E.~Theusch, J.~Rotter, M.~Medina, M.~Waterman and H.~Huang
  (2017).
\newblock Generalized correlation measure using count statistics for gene
  expression data with ordered samples.
\newblock \emph{Bioinformatics} \textbf{34}, \penalty0 617--624.

\bibitem[Watanabe(1995)]{Wata95}
S.~Watanabe (1995).
\newblock Generalized arc-sine laws for one-dimensional diffusion processes and
  random walks.
\newblock In \emph{Proceedings of Symposia in Pure Mathematics}, volume~57,
  pages 157--172.

\bibitem[Williams(1969)]{Williams69}
D.~Williams (1969).
\newblock Markov properties of {B}rownian local time.
\newblock \emph{Bull. Amer. Math. Soc.} \textbf{75}, \penalty0 1035--1036.
\newblock ISSN 0002-9904.

\end{thebibliography}

\end{document}